\documentclass[10pt]{article}
\usepackage{xspace}
\usepackage[english,french]{babel}
\usepackage[latin1]{inputenc}
\usepackage[T1]{fontenc}
\usepackage{pslatex}
\usepackage{amsmath,amssymb,amsfonts,amsthm,amscd}
\usepackage{lscape}
\usepackage{url}
\usepackage{longtable}
\usepackage{multirow}
\usepackage{graphicx}
\usepackage{hyperref}
\usepackage{listings}
\usepackage{color}
\usepackage{tikz-cd}
\usepackage{fullpage}
\usepackage{longtable}

\def\ZZ{\ensuremath{\mathbb   Z}}

\def\Q{\ensuremath{\mathbb   Q}}
\def\QQ{\ensuremath{\mathbb   Q}}

\def\CC{\ensuremath{\mathbb   C}}

\def\RR{\ensuremath{\mathbb   R}}

\def\PP{\ensuremath{\mathbb   P}}

\def\cB{\ensuremath{\mathcal B}}

\def\cE{\ensuremath{\mathcal E}}

\def\cH{\ensuremath{\mathcal H}}

\def\cL{\ensuremath{\mathcal L}}

\def\cW{\ensuremath{\mathcal W}}

\DeclareMathOperator{\Hom}{Hom}

\DeclareMathOperator{\Ext}{Ext}
\theoremstyle{plain}
\newtheorem{thm}{Théorème}[section]
\newtheorem{lem}[thm]{Lemme}

\newtheorem{prop}[thm]{Proposition}

\theoremstyle{definition}
\newtheorem{defn}[thm]{Définition}
\newtheorem{rem}{Remarque}[thm]

\usepackage{algorithm}
\usepackage[noend]{algorithmic}

\reversemarginpar
\usepackage[noblocks]{authblk}

\author[*]{Karim Belabas}
\author[**]{Dominique Bernardi}
\author[***]{Bernadette Perrin-Riou}

\affil[*]{Univ. Bordeaux, CNRS, INRIA, IMB, UMR 5251, F-33400 Talence, France}
\affil[**]{IMJ-PRG, UPMC, 4 place Jussieu, F-75005 Paris, France}
\affil[***]{Laboratoire de Mathématiques d'Orsay, Univ. Paris-Sud, CNRS, Université Paris-Saclay,
F-91405 Orsay, France.}

\title{La constante de Manin et le degré modulaire d'une courbe elliptique}

\date{\today}
\begin{document}
\maketitle

\selectlanguage{english}
\begin{abstract}
We revisit the calculation of the strong Weil curve in an isogeny class of elliptic
curves over $\QQ$, of the Manin constant and modular degree of an elliptic curve, using modular
symbols as defined in \cite{PS}, now implemented in Pari/GP. There is no innovation claim.
\end{abstract}
\selectlanguage{french}

On revisite les procédures de calcul de la courbe de Weil forte d'une classe d'isogénie
de courbes elliptiques sur $\QQ$, de la constante de Manin et du degré modulaire
d'une courbe elliptique en utilisant les symboles modulaires tels qu'ils sont
décrits dans \cite{PS} et désormais implémentés dans Pari/GP.
Il n'y a aucune prétention à une nouveauté quelconque.

\thanks{
Le troisième auteur (en particulier) remercie chaleureusement
Bill Allombert pour l'avoir lancée sur cette question, ce qui l'a obligée
à relire les articles fondateurs des années 1970 sur cette question.

Les calculs présentés à la fin de cet article ont été réalisés sur la
plateforme PLAFRIM, portée par l'action de développement INRIA PlaFRIM
avec le soutien de l'IMB, du LaBRI et d'autres entités:
Conseil Régional d'Aquitaine, FeDER, Université de Bordeaux et CNRS (voir
\url{https://www.plafrim.fr/en/home/}).
}
\tableofcontents

\section{Rappels sur les constantes de Manin et le degré modulaire}
Reprenons quelques définitions et énoncés de \cite{Mazur72}.

\begin{defn}
Une \textsl{paramétrisation modulaire} $\pi$ d'une courbe elliptique $E/\QQ$ de conducteur $N$
est un morphisme non constant $\pi_{\text{mod}}: X_0(N) \to E$
défini sur $\QQ$ tel que l'image de $\infty$ est $0$ et
tel que si $\omega$ est une forme différentielle invariante
non nulle de $E$ sur $\QQ$, l'image réciproque $\pi^*\omega$ sur $X_0(N)$
vue comme une forme modulaire parabolique est de niveau exactement $N$.
Autrement dit, il existe un rationnel $c$ non nul et une forme $f$ primitive de niveau $N$
(vecteur propre pour tous les opérateurs de Hecke de niveau $N$ et normalisée
par $a_1(f)=1$) tels que $\pi_{\text{mod}}^*\omega=c \omega_f$
avec $\omega_f=2i\pi f(z)dz \in (1+q\ZZ[[q]])\, dq$ pour $q=\exp(2 i\pi z)$.\end{defn}

\begin{defn}\begin{itemize}
\item Le \textsl{degré modulaire} $\deg(E)$ de $E$ est le degré minimal parmi
les paramétrisations modulaires de $E$.
\item Soit $\omega_E$ une forme différentielle de Néron de $E$.
La \textsl{constante de Manin} de $E$ est le rationnel $c_E > 0$
tel que $\pi_{\min,E}^* \omega_E =c_E \omega_f$ pour une paramétrisation minimale
$\pi_{\min,E}$ de $E$.
\end{itemize}
\end{defn}
Soit $\cE$ une classe d'isogénie de courbes elliptiques sur $\QQ$.
Il existe une unique courbe $E_1$ dans $\cE$ et une unique paramétrisation modulaire
$\pi_1: X_0(N) \to E_1$ (à isomorphisme près) telle que toute paramétrisation modulaire
$\pi:X_0(N) \to E$ pour $E \in \cE$ se factorise par $\pi_1$.
\begin{equation*}
  \begin{tikzcd}
    X_0(N) \arrow{r}{\pi_1} \arrow[swap]{dr}{\pi} & E_1 \arrow{d}{} \\ & E
  \end{tikzcd}
\end{equation*}
Le degré de $\pi_1: X_0(N) \to E_1$ est minimal parmi les degrés des paramétrisations
modulaires des courbes elliptiques de la classe d'isogénie $\cE$.
Si $f$ est la forme modulaire parabolique associée à $\cE$,
la courbe $E_1$ est la \textsl{courbe de Weil forte associée à $f$}.

On la construit de la manière suivante.
Soit $\pi': X_0(N) \to E'$ une paramétrisation modulaire de $E'$ dans la classe
d'isogénie $\cE$.
Par passage à la jacobienne, on en déduit un homomorphisme de variétés
abéliennes $\pi: J_0(N) \to E'$. Soit $A$ le noyau de $\pi$.
La composante connexe $A^0$ de l'élément neutre dans $A$ est une variété
abélienne qui ne dépend pas de la paramétrisation $\pi'$,
$E_1 = J_0(N)/A^0$ est une courbe elliptique et on a le diagramme commutatif
\begin{equation*}
  \begin{tikzcd}
    X_0(N) \arrow{r}{\pi_1} \arrow[swap]{dr}{\pi'} & J_0(N)/A^0=E_1 \arrow{d}{} \\ & E'
  \end{tikzcd}
\end{equation*}

\begin{lem} Soit $\pi:X_0(N) \to E$ une paramétrisation modulaire
minimale. Alors, $E$ est une courbe elliptique de Weil forte si et seulement si
$H_1(X_0(N), \ZZ) \to H_1(E,\ZZ)$ est surjective.
\end{lem}
Voir lemme 3, § 4 de \cite{Mazur72}.
\section{Espace des symboles modulaires}

\subsection{Lien entre l'espace des symboles modulaires et les cohomologies}
Soit $\Delta = \ZZ[\PP^1(\Q)]$ le groupe des diviseurs sur $\PP^1(\QQ)$
et $\Delta_0$ le sous-module des diviseurs de $\Delta$ de degré 0.
Les éléments de $\Delta_0$ peuvent être vus comme des combinaisons linéaires
dans $\ZZ$ de chemins $(a,b)$ pour $a$ et $b$ dans $\PP^1(\QQ)$.
Les modules $\Delta$ et $\Delta_0$ sont munis d'une action de $SL_2(\ZZ)$
et donc de $\Gamma=\Gamma_0(N)$.
On a une suite exacte
\begin{equation*}
0 \to \Hom_\Gamma(\ZZ, \QQ) \to \Hom_\Gamma(\Delta, \QQ)\to
\Hom_\Gamma(\Delta_0, \QQ)\to
\Ext^1_\Gamma(\ZZ,\QQ)\to \Ext^1_\Gamma(\Delta,\QQ)
\to \Ext^1_\Gamma(\Delta_0,\QQ) \to 0
\end{equation*}
ou en explicitant
\begin{equation*}
0 \to \QQ \to
\oplus_{s\in C(\Gamma)} H^0(\Gamma_s, \QQ)\to
\Hom_\Gamma(\Delta_0, \QQ)\to
H^1(\Gamma,\QQ)\to \oplus_{s\in C(\Gamma)} H^1(\Gamma_s, \QQ)
\to \Ext^1_\Gamma(\Delta_0,\QQ) \to 0
\end{equation*}
où $C(\Gamma)$ est un système de représentants de $\Gamma\backslash \PP^1(\QQ)$
et
$\Gamma_s$ le stabilisateur de $s$ dans $\Gamma$.
L'application $\Hom_\Gamma(\Delta_0, \QQ)\to H^1(\Gamma,\QQ)$
est donnée par $\Phi \mapsto
\Big(\gamma \mapsto \Phi\big((0,\gamma^{-1} 0)\big)\Big)$.
La formule des coefficients universels en cohomologie implique la suite exacte
$$ 0 \to \Ext^1(H_{0}(X_0(N),\ZZ), \ZZ)
\to H^1(X_0(N),\ZZ) \to \Hom(H_{1}(X_0(N),\ZZ),\ZZ)\to 0
$$
On en déduit que
$$H^1(X_0(N),\ZZ) = \Hom(H_1(X_0(N),\ZZ),\ZZ) $$

Rappelons les résultats de Manin (\cite{manin72}).
Si $a$ et $b$ dans $\PP^1(\QQ)$ sont équivalents modulo $\Gamma_0(N)$,
l'image de l'arc géodésique $(a,b)$ joignant $a$ à $b$ dans $\cH$
définit un élément de $H_1(X_0(N),\ZZ)$. De plus, si $a \in \PP^1(\QQ)$,
les éléments de la forme $(a,b)$ avec $b$ équivalent
à $a$ modulo $\Gamma_0(N)$ engendrent $H_1(X_0(N), \ZZ)$ (\cite{manin72}, proposition 1.4).
On peut prolonger cette application aux arcs géodésiques de la forme $(a,b)$
avec $a$ et $b$ dans $\PP^1(\QQ)$ à condition d'étendre les valeurs à
$H_1(X_0(N), \QQ)=\QQ \otimes H_1(X_0(N), \ZZ)$.
On en déduit une application surjective
$$\Delta_0 \to H_1(X_0(N),\QQ)$$
et par dualité un homomorphisme injectif
$$\Hom_\QQ(H_1(X_0(N),\QQ), \QQ) \to \Hom_\ZZ(\Delta_0, \QQ)$$
à valeurs dans les invariants par $\Gamma_0(N)$, donc
un homomorphisme injectif
$$\Hom_\QQ(H_1(X_0(N),\QQ), \QQ) \to \Hom_{\Gamma_0(N)}(\Delta_0, \QQ)$$
que l'on note provisoirement $h \to \tilde{h}$.
Son image est d'intersection nulle avec l'image de
$\Hom_{\Gamma_0(N)}(\Delta, \QQ)$. En effet, soit $h$
un élément de $\Hom_\QQ(H_1(X_0(N),\QQ), \QQ)$ tel que
$\tilde{h}$ est l'image de $\tilde{h}_1\in \Hom_{\Gamma_0(N)}(\Delta, \QQ)$. Si
$c$ est un lacet partant de $x_0 \in \PP^1(\QQ)$, son relèvement dans $\cH$
joint $x_0$ à $\gamma^{-1} x_0$ pour un $\gamma$ dans $\Gamma_0(N)$
et $h(c)=\tilde{h}\big((x_0,\gamma^{-1} x_0)\big)
 = \tilde{h}_1([\gamma^{-1} x_0])-\tilde{h}_1([x_0])=0$
puisque $\tilde{h}_1$ est invariant par $\Gamma_0(N)$.

Pour une raison de dimension (les deux espaces sont de dimension
2 fois le genre de $X_0(N)$), on en déduit les isomorphismes
$$
\begin{CD}
H^1(X_0(N),\QQ) @>\cong>> \Hom_\QQ(H_1(X_0(N),\QQ), \QQ) @>\cong>>
  \Hom_{\Gamma_0(N)}(\Delta_0, \QQ)/\Hom_{\Gamma_0(N)}(\Delta, \QQ)
\end{CD}
$$

\subsection{Structure entière de l'espace des symboles modulaires}
Remarquons que le $\ZZ$-module $H^1(\Gamma_0(N), \ZZ)$ s'injecte dans
$H^1(\Gamma_0(N), \QQ)$. Cela se voit facilement en écrivant la suite de cohomologie
pour la suite exacte $0\to \ZZ \to \QQ \to \QQ/\ZZ \to 0$
de $\Gamma_0(N)$-modules.

On note $W=\Hom_{\Gamma_0(N)}(\Delta_0,\QQ)/ \Hom_{\Gamma_0(N)}(\Delta,\QQ)$
et $W_{par}$ le supplémentaire de $\Hom_{\Gamma_0(N)}(\Delta,\QQ)$ dans $\Hom_{\Gamma_0(N)}(\Delta_0,\QQ)$
stable par l'algèbre de Hecke. Soient $\cW \subset W$ et $\cW_{par} \subset W_{par}$
les sous-$\ZZ$-modules formés des éléments dont l'image dans
$H^1(\Gamma_0(N), \QQ)$ appartient à l'image de $H^1(\Gamma_0(N), \ZZ)$ dans
$H^1(\Gamma_0(N), \QQ)$.

\begin{lem} Les $\ZZ$-modules $\cW$ et $\cW_{par}$ sont de type fini
et de rang maximal, respectivement dans $W$ et $W_{par}$.
\end{lem}
Pour une démonstration analogue, voir par exemple \cite{shimura}, Proposition 1.

\section{Espace associé à une forme modulaire et à une courbe elliptique}
\subsection{Structure rationnelle}
Soit $E$ une courbe elliptique sur $\QQ$ de niveau $N$,
$f$ la forme modulaire parabolique normalisée de poids 2 pour $\Gamma_0(N)$
qui lui est associée et $\pi : X_0(N) \to E$ une paramétrisation modulaire de
$E$.

Soit $V_f$ le $\QQ$-espace vectoriel de $W = \Hom_{\Gamma_0(N)}(\Delta_0, \QQ)$
associé obtenu par décomposition
par les opérateurs de Hecke .
Autrement dit, $V_f$ est l'ensemble des éléments $v$ de
$\Hom_{\Gamma_0(N)}(\Delta_0, \QQ)$
tel que $(T(n) - a_n(f))v=0$ pour tout entier $n\geq 1$.
Il est en fait contenu dans le sous-espace $W_{par}$ de $W$.

On a un diagramme commutatif
\begin{center}
\begin{tikzcd}
H^1(E,\QQ)\arrow{d}[swap]{\cong}\arrow{r}&H^1(X_0(N),\QQ)\arrow{d}[swap]{\cong}\\
\Hom_{\QQ}(H_1(E,\QQ),\QQ)\arrow{d}[swap]{\cong}\arrow{r} &\arrow{d}[swap]{\cong}\Hom_{\QQ}(H_1(X_0(N),\QQ), \QQ)&\\
V_f\arrow{r}& W_{par} \cong\Hom_{\Gamma_0(N)}(\Delta_0, \QQ)/\Hom_{\Gamma_0(N)}(\Delta, \QQ)
\end{tikzcd}
\end{center}

\subsection{\texorpdfstring{Structure entière $\cL_f$ associée à une forme modulaire propre}{Lg}}
Soit $\cL_f$ le $\ZZ$-sous-module de $V_f$ défini par
$\cL_f=V_f \cap \cW$. C'est donc l'ensemble des éléments $\Phi$ de $V_f$
tels que $\Phi\big((0,\gamma^{-1} 0)\big) \in \ZZ$ pour tout $\gamma \in \Gamma_0(N)$.
C'est aussi $H^1(X_0(N), \ZZ) \cap V_f$
(ici, on voit simplement $H^1(X_0(N), \ZZ)$ comme le dual sur $\ZZ$
de $H_1(X_0(N), \ZZ)$).

\subsection{\texorpdfstring{Quelques bases particulières de l'homologie de $E$}{Lg}}

Plusieurs bases de $H_1(E,\ZZ)$ sont construites. La première est définie
à l'aide des sous-espaces propres par la conjugaison complexe sur $E(\CC)$
dans un modèle minimal de Weierstrass,
la seconde est liée à la paramétrisation modulaire de $E$.

\subsubsection{Réseau associé à un modèle minimal}
Faisons quelques rappels sur les périodes complexes de $E$.
On note $c_\infty$ le nombre de composantes connexes de $E(\RR)$.
 Soit $\Lambda_E$ le réseau des périodes de $E$ dans un modèle minimal,
 autrement dit le sous-$\ZZ$-module
de rang 2 de $\CC$ formé des nombres $\int_{c} \omega_E$ pour
$c \in H_1(E,\ZZ)$. Soient $\Omega_E^+ \in \RR^+$ et
$\Omega_E^- \in i\RR^+$ tels que $\ZZ \Omega_E^+ = \Lambda_E \cap \RR$,
$\ZZ\Omega_E^- = \Lambda_E \cap i\RR$; on note $\Omega_1 = \Omega^+$ et
\footnote{La normalisation de $(\omega_1,\omega_2) = \texttt{E.omega}$
 dans Pari/GP change le signe de $\Omega_E^-$: on a $\omega_1 = \Omega_1$ et
 $$
 \omega_2 = \begin{cases}
 \frac{1}{2}(\Omega_E^+ - \Omega_E^-) = \Omega_1-\Omega_2&
   \text{si $c_\infty=1$},\\
 -\Omega_E^- = -\Omega_2&\text{si $c_\infty=2$}.
\end{cases}
 $$}
\begin{equation*}
\Omega_2 = \begin{cases}
\frac{1}{2}(\Omega_E^+ + \Omega_E^-) &\text{si $c_\infty=1$},\\
\Omega_E^- &\text{si $c_\infty=2$}.
\end{cases}
\end{equation*}
On a $\Lambda_E = \ZZ \Omega_1 + \ZZ \Omega_2$ et
$\Omega_E^+ \Omega_E^-= \frac{2}{c_\infty} \text{Aire}(\Lambda_E)$.
Notons $\delta_+$, $\delta_-$, $\delta_1$ et $\delta_2$
les éléments de $H_1(E,\ZZ)$ tels que
$$
\int_{\delta_+} \omega_E=\Omega_{E}^+,\quad
\int_{\delta_-} \omega_E=\Omega_{E}^-,\quad
\int_{\delta_1} \omega_E=\Omega_1,\quad\text{et}\quad
\int_{\delta_2} \omega_E=\Omega_2.
$$
Bien sûr, $\delta_1=\delta_+$ et
\begin{equation*}
\frac{2}{c_\infty}\delta_2=\begin{cases}
\delta_- + \delta_+&\text{si $c_\infty=1$},\\
\delta_-&\text{si $c_\infty=2$}.\\
\end{cases}
\end{equation*}
On note $(\delta_+^*, \delta_-^*)$ (resp. $(\delta_1^*, \delta_2^*)$)
la base duale de $(\delta_+, \delta_-)$ (resp.
$(\delta_1, \delta_2)$) dans $H^1(E,\ZZ)=\Hom_{\ZZ}(H_1(E,\ZZ), \ZZ)$:
$(\delta_1,\delta_2)$ est une base de $H_1(E,\ZZ)$
et $(\delta_1^*,\delta_2^*)$ est une base de $H^1(E,\ZZ)$. Explicitement,
on a
\begin{equation*}
 \delta_2^*=\frac{2}{c_\infty}\delta^*_-
 \quad\text{et}\quad
 \delta_1^* = \begin{cases}
   \delta^*_+ - \delta^*_- &\text{si $c_\infty=1$},\\
   \delta^*_+ &\text{si $c_\infty=2$}.
 \end{cases}
\end{equation*}
On en déduit l'action de la conjugaison complexe $\sigma$ sur la base
$(\delta_1^*,\delta_2^*)$
\begin{equation*}
\sigma \delta_2^*= -\delta_2^*
\quad\text{et}\quad
\begin{cases}
  \sigma \delta_1^*= \delta_1^* + \delta_2^* &\text{si $c_\infty=1$}\\
  \sigma \delta_1^*= \delta_1^* &\text{si $c_\infty=2$}.
\end{cases}
\end{equation*}

\subsubsection{Réseau associé à la paramétrisation modulaire}\label{param}
Fixons une paramétrisation modulaire minimale $\pi: X_0(N) \to E$.
On note $(n_{1,E}, n_{2,E})$ les diviseurs élémentaires de l'image $\pi_*H_1(X_0(N),
\ZZ)$ de $H_1(X_0(N), \ZZ)$ dans $H_1(E,\ZZ)$ :
il existe une base $(\delta_1, \delta_2)$ de $H_1(E,\ZZ)$
et une base $(\pi_*(\beta_1), \pi_*(\beta_2))$ de $\pi_*H_1(X_0(N), \ZZ)$
telles que
$$
\begin{cases} \pi_*(\beta_1) &= n_{1,E} \cdot \delta_1,\\
\pi_*(\beta_2)&=n_{2,E} \cdot \delta_2,
\end{cases}
$$
et les deux entiers positifs $n_{1,E}$ et $n_{2,E}$, $n_{2,E} \mid n_{1,E}$,
sont uniques.

\subsection{\texorpdfstring{Structure entière $\cL_E$ associée à une courbe elliptique}{Lg}}
Soit une base $\cB=(\delta_1,\delta_2)$ de $H_1(E, \ZZ)$.
On note $(\delta_1^*$, $\delta_2^*)$ la base de
$H^1(E, \ZZ)$ duale : on a donc pour $\delta \in H_1(E,\QQ)$
$$ \delta= \delta_1^*(\delta) \cdot \delta_1
         + \delta_2^*(\delta) \cdot \delta_2$$
et si $\omega \in H^0(E,\Omega^1_E)$
$$
\int_{\delta} \omega =
  \delta^*_1(\delta) \int_{\delta_1}\omega
+ \delta^*_2(\delta) \int_{\delta_2}\omega\,.
$$
En particulier, si $\omega_E$ est une forme différentielle de Néron
de $E$, on a
$$
\int_{\delta} \omega_E = \delta^*_1(\delta) \cdot \Omega_{E,1}
 + \delta^*_2(\delta) \cdot \Omega_{E,2}
$$
avec $\Omega_{E,i}=\int_{\delta_i} \omega_E$.
Soient $\Phi_{E,1}^{\cB}$, $\Phi_{E,2}^{\cB}$ les éléments de $V_f$ définis
pour $\beta \in H_1(X_0(N),\ZZ)$ par
\begin{equation}
\label{def_b}
\int_{\beta} \omega_f =
 \Phi_{E,1}^{\cB}(\beta) \int_{\delta_1} \omega_E
+\Phi_{E,2}^{\cB}(\beta) \int_{\delta_2} \omega_E
\end{equation}
 (seule la $f$-composante intervient).

\begin{defn}
Le sous-$\ZZ$-module de $V_f$ engendré par $\Phi_{E,1}^\cB$ et $\Phi_{E,2}^\cB$
ne dépend pas de la base choisie $\cB$. On le note $\cL_E$.
\end{defn}
\begin{lem}
Les diviseurs élémentaires (au sens généralisé) de
$\cL_E$ dans $\cL_f$ sont $(\frac{n_{1,E}}{c_E},\frac{n_{2,E}}{c_E})$.
\end{lem}
\begin{proof}
Si $n\in \ZZ$, $\beta \in H_1(X_0(N),\ZZ)$ et $\delta \in H_1(E,\ZZ)$ sont tels que
$n \cdot \delta = \pi_*(\beta)$, on a
$$n\int_{\delta} \omega_E= \int_{\pi_*\beta}\omega_E=
\int_{\beta} \pi^*\omega_E
=c_E\int_{\beta}\omega_f
$$
puisque $\pi^* \omega_E= c_E \omega_f$.
Lorsque $\cB = (\delta_1, \delta_2)$ est une base de $H_1(E,\ZZ)$ adaptée aux
diviseurs élémentaires de l'image de $H_1(X_0(N),\ZZ)$ dans $H_1(E,\ZZ)$ comme
en \ref{param}, en appliquant cela à $\beta_1$ et $\beta_2$, l'équation \eqref{def_b} devient pour tout $\beta \in H_1(X_0(N),\ZZ)$
\begin{equation*}
\begin{split}
\int_{\beta} \omega_f &=
\frac{c_E}{n_{1,E}}\Phi_{E,1}^{\cB}(\beta) \int_{\beta_1}
\omega_f + \frac{c_E}{n_{2,E}}\Phi_{E,2}^{\cB}(\beta) \int_{\beta_2}\omega_f.
\end{split}
\end{equation*}
Donc
$
\frac{c_E}{n_{1,E}}\Phi_{E,1}^{\cB},
\frac{c_E}{n_{2,E}}\Phi_{E,2}^{\cB}
$
est par définition une base de $\cL_f$.
On en déduit le lemme.
\end{proof}
Remarquons que $c_E \cL_E$ est contenu dans $\cL_f$, mais que
$\cL_E$ ne l'est pas forcément (la constante de Manin d'une courbe de Weil
forte est dans $\ZZ$ et conjecturée égale à 1,
mais pas celle d'une courbe elliptique quelconque).

\subsection{\texorpdfstring{Calcul de $\cL_E$}{Lg}}
Le calcul de $\cL_E$ est fait
à partir de bases $\delta_\pm$ des espaces propres $H^1(E,\ZZ)^{\pm}$
pour la conjugaison complexe sur $E(\CC)$.
Rappelons que si $\psi$ est un caractère de Dirichlet de conducteur
$D$ premier à $N$,
$$\tau(\psi)L(E, \overline{\psi},1)=
\sum_{a \bmod D}\psi(a)\int_{\frac{a}{D}}^\infty 2i\pi f(z) dz$$
où
$\tau(\psi)= \sum_{a \bmod D} \psi(a) e^{2 i \pi a/D}$
est la somme de Gauss de $\psi$.
Choisissons deux caractères de Dirichlet $\psi_0$ et $\psi_1$
de conducteurs respectifs $D_0$ et $D_1$, de signe
$\psi_j(-1) = \epsilon_j=(-1)^j$ et tels que $L(E,\overline{\psi_j},1)$ soit non nul pour
$j=0, 1$ (on les choisira de conducteur minimal pour cette propriété).

Les conditions
$$\tau(\psi_j)L(E, \overline{\psi}_j, 1)=
\sum_{a \bmod D_j} \psi_j(a) \cdot \Phi_{E,\epsilon_j}^{(\delta_+, \delta_-)}
\big((a/D_j, \infty)\big)
\int_{\delta_{\epsilon_j}}\omega_E$$
pour $j=0,1$ déterminent les symboles modulaires $\Phi_{E,\pm}^{(\delta_+, \delta_-)}$
dans le $\QQ$-espace vectoriel $V_f^{\pm}$ de dimension 1 de manière indépendante du
choix des $\psi_j$ :
cela revient en effet à imposer les conditions
$$
\int_{\beta_0} \omega_f =
 \Phi_{E,+}^{(\delta_+, \delta_-)}(\beta_0) \int_{\delta_+} \omega_E,
$$
$$
\int_{\beta_1} \omega_f =
\Phi_{E,-}^{(\delta_+, \delta_-)}(\beta_1) \int_{\delta_-} \omega_E,
$$
avec $\beta_j=\sum_{a \bmod D_j} \psi_j(a) \big((\frac{a}{D_j},\infty)\big)$.
Par exemple, pour $\Psi_0$ trivial et $\beta_0=(0, \infty)$,
la formule devient
$$
\frac{L(E,1)}{\Omega_{E}^+} = \Phi_{E,+}^{(\delta_+,
\delta_-)}\big((0,\infty)\big).
$$
Le réseau $\cL_E$ de $V_f$ associé à $E$ est alors engendré par
\begin{equation*}
\begin{cases}
\Phi_{E,+}^{(\delta_+, \delta_-)} \quad\text{et}\quad
\Phi_{E,-}^{(\delta_+, \delta_-)}& \text{si $c_\infty=2$}\\
\Phi_{E,+}^{(\delta_+, \delta_-)}-\Phi_{E,-}^{(\delta_+, \delta_-)}
\quad\text{et}\quad
2\Phi_{E,-}^{(\delta_+, \delta_-)}&\text{si $c_\infty=1$}
\end{cases}
\end{equation*}
\begin{rem}
Il est avantageux pour les calculs d'utiliser des caractères de Dirichlet
de conducteur le plus petit possible. Il parait naturel d'utiliser un
caractère quadratique mais ce n'est pas toujours possible ni optimal. Par exemple,
soit la courbe $E=1225e1$ de conducteur $N=1225=35^2$. Le signe de l'équation
fonctionnelle de $E$ est $1$. Montrons en utilisant Pari/GP que le signe de
l'équation fonctionnelle des twists de $E$ par un caractère quadratique impair
$\psi$ est $-1$, ce qui impliquera que $L(E^{(\psi)},1)=0$ pour tout $D$
négatif. Pour $D$ discriminant fondamental, on note $E^{(D)}$ le twist de
$E$ par le caractère quadratique $(D/.)$, $w(E)$ le signe de l'équation
fonctionnelle de $E$. On a la formule bien connue $w(E^{(D)})=\psi(-N)w(E)$
pour $D$ premier à $N$, donc ici $w(E^{(D)}) =\psi(-1)w(E)
=\text{signe}(D)w(E)$ pour $(D,35) = 1$.
\begin{verbatim}
? E = ellinit("1225e1");
? ellrootno(E)
%2 = 1

? E7 = ellinit(elltwist(E,-7)); ellidentify(E7)[1][1]
%3 = "1225g1"
? ellrootno(E7)
%4 = -1

? E5 = ellinit(elltwist(E,5)); ellidentify(E5)[1][1]
%5 = "1225f1"
? ellrootno(E5)
%6 = 1

? E35 = ellinit(elltwist(E,-35)); ellidentify(E35)[1][1]
%7 = "1225h1"
? ellrootno(E35)
%8 = -1
\end{verbatim}
Ainsi, les twists de $E$ par $-7$, $5$, $-35$ sont encore de conducteur un
carré.
\begin{itemize}
\item Si $D=-7m$ avec $m$ premier à $35$ et positif,
$$w(E^{(D)}) = w({E^{(-7)}}^{(m)}) = \text{signe}(m) w(E^{(-7)})
  = w(E^{(-7)}) = -1.$$
\item Si $D=5m$ avec $m$ premier à $35$ et négatif,
$$w(E^{(D)}) = w({E^{(5)}}^{(m)}) = \text{signe}(m) w(E^{(5)})
  = -w(E^{(5)}) = -1.$$
\item Si $D=35m$ avec $m$ premier à $35$ et positif,
$$w(E^{(D)}) = w({E^{(-35)}}^{(m)}) = \text{signe}(m) w(E^{(-35)})
  = w(E^{(-35)}) = -1.$$
\end{itemize}
Autrement dit, $L(E,\psi, 1)$
est nul pour tous les caractères quadratiques impairs.
On utilise alors pour le calcul de $\Phi_{E,-}^{(\delta_+, \delta_-)}$
un caractère de Dirichlet impair de conducteur 9 et d'ordre 6.
\end{rem}

\section{Conséquences}
\subsection{Critère de calcul de la courbe de Weil forte et de sa constante de Manin}
\begin{prop}
\begin{enumerate}
\item
La courbe $E$ est une courbe de Weil forte si et seulement si
les sous-$\ZZ$-modules $\cL_f$ et $\cL_E$ de $V_f \subset W$ sont homothétiques.
\item Si $\cL_E=\cL_f$, la courbe $E$ est la courbe de Weil forte
et la constante de Manin est égale à 1.
\end{enumerate}
\end{prop}
\begin{proof}
La courbe de Weil forte d'une classe d'isogénie sur $\QQ$
de courbes elliptiques est la courbe elliptique
telle que $\pi_*: H_1(X_0(N),\ZZ) \to H_1(E,\ZZ)$ est surjective.
Avec les notations précédentes, cela est le cas si et seulement
$n_{1,E}=n_{2,E}=1$.
Les deux sous-$\ZZ$-modules $\cL_E$ et $\cL_f$ de $V_f$ sont alors homothétiques.
Réciproquement, notons $E_1$ la courbe de Weil forte
(dont on sait qu'elle existe et vérifie que $\cL_{E_1}$ est homothétique à $\cL_f$)
et $\alpha$ une isogénie de $E$ sur $E_1$
de noyau cyclique. Les diviseurs élémentaires
de $\cL_{E_1}$ dans $\cL_{E}$ sont alors $\deg(\alpha)$ et $1$.
Si l'on suppose ces deux réseaux homothétiques, $\alpha$ est un isomorphisme.
Cela démontre la première partie.

Supposons maintenant que $\cL_E=\cL_f$. On a alors $n_{1,E}=n_{2,E}=c_E$.
Mais il n'y a qu'une seule courbe de Weil forte à isomorphisme près.
Donc, on a nécessairement $n_{1,E}=n_{2,E}=c_E=1$.
\end{proof}

\subsection{Graphe orienté des isogénies admissibles}
Prenons une classe d'isogénie de courbes elliptiques sur $\QQ$.
Le graphe des isogénies de degré premier peut être orienté de la manière suivante.
Si $\pi:E \to E'$ est une isogénie de degré premier
et $\pi'$ l'isogénie duale, on a
 $\pi^*\omega_{E'}=\pm \omega_{E}$ ou ${\pi'}^*\omega_{E}=\pm \omega_{E'}$.
L'arc entre $E$ et $E'$ est alors orienté comme $E\to E'$
si et seulement si $\pi^*\omega_{E'}=\pm \omega_{E}$
(appelé \emph{admissible isogeny} dans \cite{manin72}, 6.11).
Le calcul de ce graphe peut se faire par un calcul des modèles minimaux
des courbes elliptiques de la classe d'isogénie et des réseaux des périodes complexes
associés.

\subsection{Degré modulaire et constante de Manin} Soit $E_1$ une courbe de Weil forte et
$E$ une courbe elliptique sur $\QQ$ dans sa classe d'isogénie.
Une paramétrisation modulaire minimale $X_0(N) \to E$
est obtenue comme composé $X_0(N) \to E_1 \to E$
où $E_1\to E$ est de degré minimal.
Le degré modulaire de $E$ est égal au produit du
degré modulaire de $E_1$ par le degré de cette isogénie.
D'autre part, le degré modulaire de la courbe de Weil forte
$E_1$ se calcule par la formule
\footnote{
Si l'on désire rester dans le cadre algébrique,
ce degré peut aussi être calculé à partir de la matrice de Gram
sur une base de $\cL_f$ du produit de Petersson algébrique
(ou produit d'intersection) défini sur l'espace
$\Hom_{\Gamma_0(N)}(\Delta_0, \Q)$ (voir article à venir).
}
\begin{equation}
\label{degmod}
4 \pi^2 ||f||^2=\deg(\pi_1) \text{Aire}(E_1)
=\frac{c_\infty}{2} \deg(\pi_1) \Omega_{E_1}^+\Omega_{E_1}^-
\end{equation}
une fois qu'on a vérifié que la constante de Manin de $E_1$ est bien 1.
On obtient l'algorithme suivant.
\begin{algorithm}[H]
\caption{Degré modulaire et constante de Manin}
\label{algo:moddegree}
\begin{algorithmic}[1]
\REQUIRE Une courbe elliptique $E$.
\ENSURE Le degré modulaire $\deg(E)$ et la constante de Manin $c_E$.
\STATE On calcule la courbe de Weil forte $E_1$ dans la classe d'isogénie de $E$
et on vérifie en même temps que $c_{E_1}=1$. On calcule son degré modulaire
$d_1$ (par exemple par \texttt{ellmoddegree} dans Pari/GP).
\STATE On calcule le graphe orienté des isogénies de degré premier de la classe d'isogénie.
Si $a$ est un arc, on note $\deg(a)$ le degré (premier) de
l'isogénie.
\STATE On cherche le chemin le plus court de $E_1$ vers $E$
dans le graphe non orienté associé.
\STATE $d \leftarrow d_1$; $c \leftarrow 1$;
\FOR{$a$ arc du chemin}
  \STATE $d \leftarrow d \times \deg(a)$
  \IF{$a$ est orienté négativement}
    \STATE $c \leftarrow c \times \deg(a)$
  \ENDIF
\ENDFOR
\RETURN $d$ et $c$.
\end{algorithmic}
\end{algorithm}

Le graphe orienté calculé a une source (courbe optimale de Stevens, conjecturalement
courbe de Weil forte relativement à $X_1(N)$, \cite{stevensopt}).
Si la courbe de Weil forte est la source du graphe orienté que l'on a construit,
la constante de Manin pour une courbe elliptique quelconque isogène
est toujours $1$.

Prenons le cas où $N=11$.
La courbe de Weil forte est "11a1". Le graphe orienté des isogénies de degré
premier est
$$11a3 \overset{5}{\to} 11a1 \overset{5}{\to} 11a2$$
La courbe optimale au sens de Stevens est "11a3". On a
\begin{verbatim}
? ellmoddegree(ellinit("11a1"))[1]
%1 = 1
\end{verbatim}

\begin{center}\begin{tabular}{|c|c|c|c|c|c|}
   \hline&deg&$c_E$&$[L_f:\cL_E]$&$[H_1(E,\ZZ):\pi^*H_1(X_0(N),\ZZ)]$&Matrice\\\hline11a1&1&1&1,1&1,1&\multirow{3}{*}{$\begin{pmatrix}0&5&0\\0&0&0\\5&0&0\end{pmatrix}$}\\*
11a2&5&1&5,1&5,1&\\*
11a3&5&5&1,1/5&5,1&\\*
\hline\end{tabular}\\
\begin{equation*}
\begin{tikzcd}
11a3 \arrow{r}{5} & \color{red}{11a1} \arrow{r}{5} & 11a2
\end{tikzcd}
\end{equation*}
\end{center}

\newpage
\appendix
\section{Programmes}
Ces programmes sont écrits en PARI/GP \cite {pari}.
\lstset{basicstyle=\ttfamily, morekeywords={my,for,if,return,next}}
\begin{lstlisting}
/* plus court chemin entre x et y sans tenir compte de l'orientation */
dj(A,x,y) =
{ my(v, w, n = #A);
  v = vector(n,i,n); v[x] = 0;
  w = vector(n);
  for (l=1, n-1, for (z=1, n,
    if (v[z] != l-1, next);
    for (t=1, n, if (A[z,t] && v[t]>l, v[t]=l; w[t]=z))));
  /* Dijkstra: v[t] = d(x,t), w[t] = dernier point sur le chemin */
  my (m, res);
  m = v[y]+1; if (m > n, return([])); /* non connectés */
  res = vector(m);
  for (i=1, m, res[m+1-i] = y; y = w[y]);
  return(res);
}

/* SNF d'une matrice à entrées rationnelles */
Qsnf(M) = my(C = content(M)); C*matsnf(M/C);

/* matrice A du graphe d'isogénie de e, indices des réseaux associés dans L_f */
graphisogeny(e) =
{ my([Isog,isomat] = ellisomat(e,1));
  Isog = [ ellminimalmodel(ellinit(e)) | e <- Isog ];
  my([M,xpm] = msfromell(Isog));
  my(Lf, vL, Lindice, A, n);
  vL = [x[3] | x<-xpm ]; /* vecteur des réseaux L[i] */
  Lf = mslattice(M, vL[1]); /* réseau associé à la courbe de Weil dans Isog*/
  Lindice = [ Qsnf(Lf^(-1)*L) | L<-vL ]; /* compare L[i] avec Lf */
  Isog = [ ellidentify(e)[1][1] | e<-Isog ];
  n = #Isog; A = matrix(n,n);
  for (i=1, n-1, for (j=i+1, n,
    if (!isprime(isomat[i,j]), next);
    my(C = Qsnf(vL[i]^(-1) * vL[j]), D = C[1]*C[2]);
    if (denominator(C)==1, A[i,j] = D
                         , A[j,i] = 1/D)));
  return([A, Isog, Lindice]);
};

ellmanin(e) =
{ my(i, i1, pc, c = 1, d, n);
  my([A, vE, Lindice] = graphisogeny(e)); /* matrice du graphe avec poids */
  A = A - mattranspose(A);
  /* i1: numéro de la courbe de Weil forte e1 dans la classe d'isogénie
   * (vérifie aussi que la constante de Manin de courbe Weil forte est 1) */
  n = #vE;
  [i1] = [ i | i<-[1..n], Lindice[i] == [1,1] ];
  e1 = ellinit(vecextract(vE,i1)[1]);
  ide = ellidentify(e)[1][1]; /* label de Cremona de la courbe */
  [i] = [ i | i<-[1..n], vE[i] == ide ]; /* numero de la courbe */
  pc = dj(A,i1,i); /* chemin le plus court entre i1 et i */
  /* calcule le degré d entre e et e1 et le multiplie par ellmoddeg(e1) */
  /* calcule la constante de Manin c de e */
  [d] = ellmoddegree(e1);
  for (j=1, #pc-1,
    my (t = A[ pc[j], pc[j+1] ], T = abs(t));
    d *= T; if (t < 0, c *= T));
  return([d, c, Lindice[i]]);
}
\end{lstlisting}
\newpage
\section{Quelques exemples}
Nous avons fait une sélection de quelques exemples.
Ainsi, nous donnons des exemples de courbes elliptiques ayant
beaucoup de courbes isogènes et pour chaque valeur possible de degré d'isogénie
($2, 3, 5, 7, 13, 1, 17, 19, 37, 43, 67,163$, voir \cite{mazur78}),
La construction des graphes a été inspirée des programmes de F. Brunault
(\cite{brunault}).
Nous renvoyons à \cite{stein-watkins}, \S 4, pour une analyse
de familles de courbes où la courbe optimale pour $X_0(N)$ n'est pas
la source de l'arbre.
\smallskip

Le dernier exemple en niveau 130050 est un cas laissé ouvert par Cremona
\cite{cremonagit}. Il faut environ 26h de calcul pour confirmer que la
courbe optimale est bien \texttt{130050em1}. Comme l'indique Cremona,
une classe d'isogénie fixée dans l'intervalle couvert par ses tables
(conducteur inférieur à 400000) peut être raisonablement certifiée en
quelques jours par le calcul complet de l'espace des symboles modulaires
comme nous l'avons fait ci-dessus. Quand le niveau $N$ augmente, $100\%$ du
temps de calcul dans les programmes de l'annexe A est utilisé par la commande
\texttt{msfromell}, qui calcule l'espace des symboles modulaires de poids $2$
et niveau $N$ puis le symbole normalisé attaché à la classe d'isogénie
considéré. La complexité asymptotique de cette implantation est
$\tilde{O}(N^{\log_2 7})$.

\begin{center}\begin{tabular}{|c|c|c|c|c|c|}
   \hline&deg&$c_E$&$[L_f:\cL_E]$&$[H_1(E,\ZZ):\pi^*H_1(X_0(N),\ZZ)]$&Matrice\\\hline11a1&1&1&1,1&1,1&\multirow{3}{*}{$\begin{pmatrix}0&5&0\\0&0&0\\5&0&0\end{pmatrix}$}\\*
11a2&5&1&5,1&5,1&\\*
11a3&5&5&1,1/5&5,1&\\*
\hline\end{tabular}\\
\begin{equation*}
\begin{tikzcd}
11a3 \arrow{r}{5} & \color{red}{11a1} \arrow{r}{5} & 11a2
\end{tikzcd}
\end{equation*}
\end{center}
\begin{center}\begin{tabular}{|c|c|c|c|c|c|}
   \hline&deg&$c_E$&$[L_f:\cL_E]$&$[H_1(E,\ZZ):\pi^*H_1(X_0(N),\ZZ)]$&Matrice\\\hline14a1&1&1&1,1&1,1&\multirow{6}{*}{$\begin{pmatrix}0&0&3&2&0&0\\3&0&0&0&2&0\\0&0&0&0&0&2\\0&0&0&0&0&3\\0&0&0&3&0&0\\0&0&0&0&0&0\end{pmatrix}$}\\*
14a4&3&3&1,1/3&3,1&\\*
14a3&3&1&3,1&3,1&\\*
14a2&2&1&2,1&2,1&\\*
14a6&6&3&2,1/3&6,1&\\*
14a5&6&1&6,1&6,1&\\*
\hline\end{tabular}\\
\begin{equation*}
\begin{tikzcd}
14a4 \arrow{r}{3} \arrow{dr}{2} & \color{red}{14a1} \arrow{r}{3}
  \arrow{dr}{2} & 14a3 \arrow{dr}{2} \\
&14a6 \arrow{r}{3} & 14a2 \arrow{r}{3} & 14a5
\end{tikzcd}
\end{equation*}
\end{center}
\begin{center}\begin{tabular}{|c|c|c|c|c|c|}
   \hline&deg&$c_E$&$[L_f:\cL_E]$&$[H_1(E,\ZZ):\pi^*H_1(X_0(N),\ZZ)]$&Matrice\\\hline15a1&1&1&1,1&1,1&\multirow{8}{*}{$\begin{pmatrix}0&0&0&0&2&2&0&0\\2&0&2&0&0&0&0&0\\0&0&0&0&0&0&0&0\\0&2&0&0&0&0&0&0\\0&0&0&0&0&0&0&0\\0&0&0&0&0&0&2&2\\0&0&0&0&0&0&0&0\\0&0&0&0&0&0&0&0\end{pmatrix}$}\\*
15a3&2&2&1,1/2&2,1&\\*
15a7&4&2&2,1/2&4,1&\\*
15a8&4&4&1,1/4&4,1&\\*
15a4&2&1&2,1&2,1&\\*
15a2&2&1&2,1&2,1&\\*
15a6&4&1&4,1&4,1&\\*
15a5&4&1&4,1&4,1&\\*
\hline\end{tabular}\\
\begin{equation*}
\begin{tikzcd}
15a8 \arrow{r}{2}& 15a3 \arrow{r}{2}\arrow{rd}{2}& \color{red}{15a1}\arrow{rd}{2}\arrow{r}{2}& 15a2\arrow{rd}{2}\arrow{r}{2}& 15a6\\
&& 15a7&15a4&15a5\end{tikzcd}
\end{equation*}
\end{center}
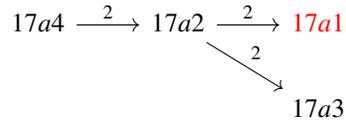
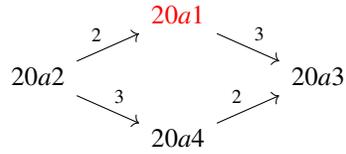
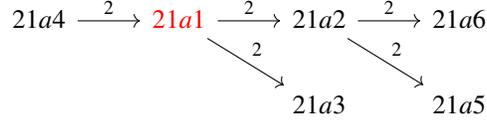
\begin{center}\begin{tabular}{|c|c|c|c|c|c|}
   \hline&deg&$c_E$&$[L_f:\cL_E]$&$[H_1(E,\ZZ):\pi^*H_1(X_0(N),\ZZ)]$&Matrice\\\hline17a1&1&1&1,1&1,1&\multirow{4}{*}{$\begin{pmatrix}0&0&0&0\\2&0&0&2\\0&2&0&0\\0&0&0&0\end{pmatrix}$}\\*
17a2&2&2&1,1/2&2,1&\\*
17a4&4&4&1,1/4&4,1&\\*
17a3&4&2&2,1/2&4,1&\\*
\hline\end{tabular}\\
\begin{equation*}
\begin{tikzcd}
17a4\arrow{r}{2} & 17a2 \arrow{rd}{2}\arrow{r}{2} & \color{red}{17a1}\\
&&17a3\end{tikzcd}
\end{equation*}
\end{center}
\begin{center}\begin{tabular}{|c|c|c|c|c|c|}
   \hline&deg&$c_E$&$[L_f:\cL_E]$&$[H_1(E,\ZZ):\pi^*H_1(X_0(N),\ZZ)]$&Matrice\\\hline20a1&1&1&1,1&1,1&\multirow{4}{*}{$\begin{pmatrix}0&3&0&0\\0&0&0&0\\2&0&0&3\\0&2&0&0\end{pmatrix}$}\\*
20a3&3&1&3,1&3,1&\\*
20a2&2&2&1,1/2&2,1&\\*
20a4&6&2&3,1/2&6,1&\\*
\hline\end{tabular}\\
\begin{equation*}
\begin{tikzcd}[row sep=small]
&\color{red}{20a1}\arrow{rd}{3} \\
20a2\arrow{ru}{2}\arrow{rd}{3}&& 20a3 \\
&20a4\arrow{ru}{2} & \end{tikzcd}
\end{equation*}
\end{center}
\begin{center}\begin{tabular}{|c|c|c|c|c|c|}
   \hline&deg&$c_E$&$[L_f:\cL_E]$&$[H_1(E,\ZZ):\pi^*H_1(X_0(N),\ZZ)]$&Matrice\\\hline21a1&1&1&1,1&1,1&\multirow{6}{*}{$\begin{pmatrix}0&2&0&2&0&0\\0&0&0&0&0&0\\2&0&0&0&0&0\\0&0&0&0&2&2\\0&0&0&0&0&0\\0&0&0&0&0&0\end{pmatrix}$}\\*
21a3&2&1&2,1&2,1&\\*
21a4&2&2&1,1/2&2,1&\\*
21a2&2&1&2,1&2,1&\\*
21a6&4&1&4,1&4,1&\\*
21a5&4&1&4,1&4,1&\\*
\hline\end{tabular}\\
\begin{equation*}
\begin{tikzcd}21a4 \arrow{r}{2}  & \color{red}{21a1} \arrow{r}{2}
   \arrow{dr}{2}& 21a2\arrow{r}{2} \arrow{rd}{2}& 21a6\\
&& 21a3&21a5\end{tikzcd}
\end{equation*}
\end{center}
\begin{center}\begin{tabular}{|c|c|c|c|c|c|}
   \hline&deg&$c_E$&$[L_f:\cL_E]$&$[H_1(E,\ZZ):\pi^*H_1(X_0(N),\ZZ)]$&Matrice\\\hline26b1&2&1&1,1&1,1&\multirow{2}{*}{$\begin{pmatrix}0&7\\0&0\end{pmatrix}$}\\*
26b2&14&1&7,1&7,1&\\*
\hline\end{tabular}\\
\begin{equation*}
\begin{tikzcd}
\color{red}{26b1} \arrow{r}{7} & 26b2
\end{tikzcd}
\end{equation*}
\end{center}
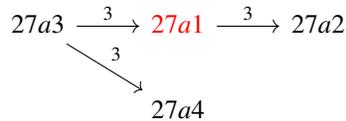
\begin{center}\begin{tabular}{|c|c|c|c|c|c|}
   \hline&deg&$c_E$&$[L_f:\cL_E]$&$[H_1(E,\ZZ):\pi^*H_1(X_0(N),\ZZ)]$&Matrice\\\hline27a1&1&1&1,1&1,1&\multirow{4}{*}{$\begin{pmatrix}0&0&0&3\\3&0&3&0\\0&0&0&0\\0&0&0&0\end{pmatrix}$}\\*
27a3&3&3&1,1/3&3,1&\\*
27a4&9&3&3,1/3&9,1&\\*
27a2&3&1&3,1&3,1&\\*
\hline\end{tabular}\\
\begin{equation*}
\begin{tikzcd}27a3 \arrow{r}{3} \arrow{rd}{3} &\color{red}{27a1}\arrow{r}{3} & 27a2\\&27a4\end{tikzcd}\end{equation*}\end{center}
\begin{center}\begin{tabular}{|c|c|c|c|c|c|}
   \hline&deg&$c_E$&$[L_f:\cL_E]$&$[H_1(E,\ZZ):\pi^*H_1(X_0(N),\ZZ)]$&Matrice\\\hline30a1&2&1&1,1&1,1&\multirow{8}{*}{$\begin{pmatrix}0&2&0&0&3&0&0&0\\0&0&2&2&0&3&0&0\\0&0&0&0&0&0&3&0\\0&0&0&0&0&0&0&3\\0&0&0&0&0&2&0&0\\0&0&0&0&0&0&2&2\\0&0&0&0&0&0&0&0\\0&0&0&0&0&0&0&0\end{pmatrix}$}\\*
30a2&4&1&2,1&2,1&\\*
30a5&8&1&4,1&4,1&\\*
30a4&8&1&4,1&4,1&\\*
30a3&6&1&3,1&3,1&\\*
30a6&12&1&6,1&6,1&\\*
30a8&24&1&12,1&12,1&\\*
30a7&24&1&12,1&12,1&\\*
\hline\end{tabular}\\
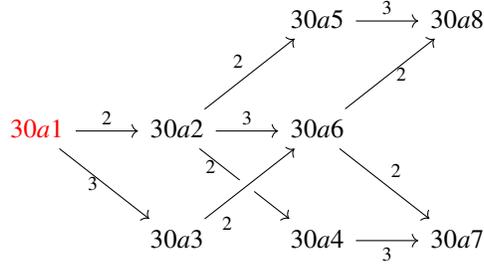
\begin{equation*}
\begin{tikzcd}[row sep=large]
 && 30a5\arrow{r}{3}&30a8 \\
\color{red}{30a1} \arrow{r}{2}\arrow{rd}[left]{3}&30a2\arrow{rd}[left,near start]{2}\arrow{ru}{2}\arrow{r}{3}&30a6\arrow{rd}{2}\arrow{ru}[right]{2}\\
&30a3\arrow[crossing over]{ru}[below right, very near start]{2}&30a4\arrow{r}[below]{3}&30a7
\end{tikzcd}
\end{equation*}
\end{center}
\begin{center}\begin{tabular}{|c|c|c|c|c|c|}
   \hline&deg&$c_E$&$[L_f:\cL_E]$&$[H_1(E,\ZZ):\pi^*H_1(X_0(N),\ZZ)]$&Matrice\\\hline32a1&1&1&1,1&1,1&\multirow{4}{*}{$\begin{pmatrix}0&0&0&0\\2&0&2&2\\0&0&0&0\\0&0&0&0\end{pmatrix}$}\\*
32a2&2&2&1,1/2&2,1&\\*
32a4&4&2&2,1/2&4,1&\\*
32a3&4&2&2,1/2&4,1&\\*
\hline\end{tabular}\\
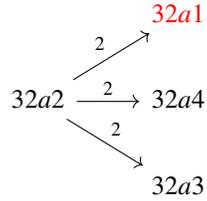
\begin{equation*}
\begin{tikzcd}
 & \color{red}{32a1} \\
32a2\arrow{ru}{2}\arrow{r}{2}\arrow{rd}{2} & 32a4\\
 & 32a3\\
\end{tikzcd}
\end{equation*}
\end{center}
\begin{center}\begin{tabular}{|c|c|c|c|c|c|}
   \hline&deg&$c_E$&$[L_f:\cL_E]$&$[H_1(E,\ZZ):\pi^*H_1(X_0(N),\ZZ)]$&Matrice\\\hline37b1&2&1&1,1&1,1&\multirow{3}{*}{$\begin{pmatrix}0&0&3\\3&0&0\\0&0&0\end{pmatrix}$}\\*
37b3&6&3&1,1/3&3,1&\\*
37b2&6&1&3,1&3,1&\\*
\hline\end{tabular}\\
\begin{equation*}
\begin{tikzcd}
37b3 \arrow{r}{3} & \color{red}{37b1} \arrow{r}{3} & 37b2
\end{tikzcd}
\end{equation*}
\end{center}
\begin{center}\begin{tabular}{|c|c|c|c|c|c|}
   \hline&deg&$c_E$&$[L_f:\cL_E]$&$[H_1(E,\ZZ):\pi^*H_1(X_0(N),\ZZ)]$&Matrice\\\hline38b1&2&1&1,1&1,1&\multirow{2}{*}{$\begin{pmatrix}0&5\\0&0\end{pmatrix}$}\\*
38b2&10&1&5,1&5,1&\\*
\hline\end{tabular}\\
\begin{equation*}
\begin{tikzcd}
\color{red}{38b1} \arrow{r}{5} & 38b2
\end{tikzcd}
\end{equation*}
\end{center}
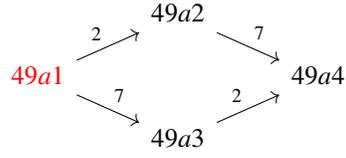
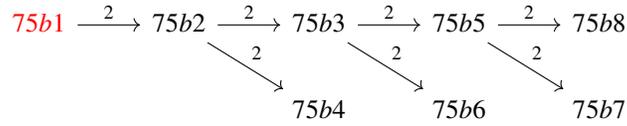
\begin{center}\begin{tabular}{|c|c|c|c|c|c|}
   \hline&deg&$c_E$&$[L_f:\cL_E]$&$[H_1(E,\ZZ):\pi^*H_1(X_0(N),\ZZ)]$&Matrice\\\hline49a1&1&1&1,1&1,1&\multirow{4}{*}{$\begin{pmatrix}0&2&7&0\\0&0&0&7\\0&0&0&2\\0&0&0&0\end{pmatrix}$}\\*
49a2&2&1&2,1&2,1&\\*
49a3&7&1&7,1&7,1&\\*
49a4&14&1&14,1&14,1&\\*
\hline\end{tabular}\\
\begin{equation*}
\begin{tikzcd}[row sep=small]
&49a2\arrow{rd}{7} \\
\color{red}{49a1}\arrow{ru}{2}\arrow{rd}{7}&& 49a4 \\
&49a3\arrow{ru}{2} & \end{tikzcd}
\end{equation*}
\end{center}
\begin{center}\begin{tabular}{|c|c|c|c|c|c|}
   \hline&deg&$c_E$&$[L_f:\cL_E]$&$[H_1(E,\ZZ):\pi^*H_1(X_0(N),\ZZ)]$&Matrice\\\hline75b1&6&1&1,1&1,1&\multirow{8}{*}{$\begin{pmatrix}0&2&0&0&0&0&0&0\\0&0&2&2&0&0&0&0\\0&0&0&0&0&0&0&0\\0&0&0&0&2&2&0&0\\0&0&0&0&0&0&0&0\\0&0&0&0&0&0&2&2\\0&0&0&0&0&0&0&0\\0&0&0&0&0&0&0&0\end{pmatrix}$}\\*
75b2&12&1&2,1&2,1&\\*
75b4&24&1&4,1&4,1&\\*
75b3&24&1&4,1&4,1&\\*
75b6&48&1&8,1&8,1&\\*
75b5&48&1&8,1&8,1&\\*
75b8&96&1&16,1&16,1&\\*
75b7&96&1&16,1&16,1&\\*
\hline\end{tabular}\\
\begin{equation*}
\begin{tikzcd}
\color{red}{75b1} \arrow{r}{2}& 75b2 \arrow{r}{2}\arrow{rd}{2}& 75b3\arrow{rd}{2}\arrow{r}{2}& 75b5\arrow{rd}{2}\arrow{r}{2}& 75b8\\
&& 75b4&75b6&75b7\end{tikzcd}
\end{equation*}
\end{center}
\begin{center}\begin{tabular}{|c|c|c|c|c|c|}
   \hline&deg&$c_E$&$[L_f:\cL_E]$&$[H_1(E,\ZZ):\pi^*H_1(X_0(N),\ZZ)]$&Matrice\\\hline121a1&6&1&1,1&1,1&\multirow{2}{*}{$\begin{pmatrix}0&11\\0&0\end{pmatrix}$}\\*
121a2&66&1&11,1&11,1&\\*
\hline\end{tabular}\\
\begin{equation*}
\begin{tikzcd}
\color{red}{121a1} \arrow{r}{11} & 121a2
\end{tikzcd}
\end{equation*}
\end{center}
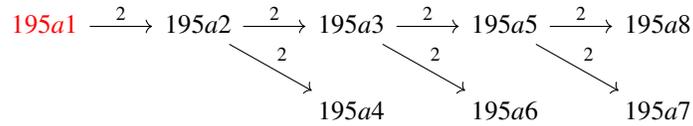
\begin{center}\begin{tabular}{|c|c|c|c|c|c|}
   \hline&deg&$c_E$&$[L_f:\cL_E]$&$[H_1(E,\ZZ):\pi^*H_1(X_0(N),\ZZ)]$&Matrice\\\hline195a1&24&1&1,1&1,1&\multirow{8}{*}{$\begin{pmatrix}0&2&0&0&0&0&0&0\\0&0&2&2&0&0&0&0\\0&0&0&0&0&0&0&0\\0&0&0&0&2&2&0&0\\0&0&0&0&0&0&0&0\\0&0&0&0&0&0&2&2\\0&0&0&0&0&0&0&0\\0&0&0&0&0&0&0&0\end{pmatrix}$}\\*
195a2&48&1&2,1&2,1&\\*
195a4&96&1&4,1&4,1&\\*
195a3&96&1&4,1&4,1&\\*
195a6&192&1&8,1&8,1&\\*
195a5&192&1&8,1&8,1&\\*
195a8&384&1&16,1&16,1&\\*
195a7&384&1&16,1&16,1&\\*
\hline\end{tabular}\\
\begin{equation*}
\begin{tikzcd}
\color{red}{195a1} \arrow{r}{2}& 195a2 \arrow{r}{2}\arrow{rd}{2}& 195a3\arrow{rd}{2}\arrow{r}{2}& 195a5\arrow{rd}{2}\arrow{r}{2}& 195a8\\
&& 195a4&195a6&195a7\end{tikzcd}
\end{equation*}
\end{center}
\begin{center}\begin{tabular}{|c|c|c|c|c|c|}
   \hline&deg&$c_E$&$[L_f:\cL_E]$&$[H_1(E,\ZZ):\pi^*H_1(X_0(N),\ZZ)]$&Matrice\\\hline208d1&48&1&1,1&1,1&\multirow{2}{*}{$\begin{pmatrix}0&7\\0&0\end{pmatrix}$}\\*
208d2&336&1&7,1&7,1&\\*
\hline\end{tabular}\\
\begin{equation*}
\begin{tikzcd}
\color{red}{208d1} \arrow{r}{7} & 208d2
\end{tikzcd}
\end{equation*}
\end{center}
\begin{center}\begin{tabular}{|c|c|c|c|c|c|}
   \hline&deg&$c_E$&$[L_f:\cL_E]$&$[H_1(E,\ZZ):\pi^*H_1(X_0(N),\ZZ)]$&Matrice\\\hline294a1&84&1&1,1&1,1&\multirow{2}{*}{$\begin{pmatrix}0&7\\0&0\end{pmatrix}$}\\*
294a2&588&1&7,1&7,1&\\*
\hline\end{tabular}\\
\begin{equation*}
\begin{tikzcd}
\color{red}{294a1} \arrow{r}{7} & 294a2
\end{tikzcd}
\end{equation*}
\end{center}
\begin{center}\begin{tabular}{|c|c|c|c|c|c|}
   \hline&deg&$c_E$&$[L_f:\cL_E]$&$[H_1(E,\ZZ):\pi^*H_1(X_0(N),\ZZ)]$&Matrice\\\hline361a1&20&1&1,1&1,1&\multirow{2}{*}{$\begin{pmatrix}0&19\\0&0\end{pmatrix}$}\\*
361a2&380&1&19,1&19,1&\\*
\hline\end{tabular}\\
\begin{equation*}
\begin{tikzcd}
\color{red}{361a1} \arrow{r}{19} & 361a2
\end{tikzcd}
\end{equation*}
\end{center}
\begin{center}\begin{tabular}{|c|c|c|c|c|c|}
   \hline&deg&$c_E$&$[L_f:\cL_E]$&$[H_1(E,\ZZ):\pi^*H_1(X_0(N),\ZZ)]$&Matrice\\\hline464e1&60&1&1,1&1,1&\multirow{2}{*}{$\begin{pmatrix}0&0\\2&0\end{pmatrix}$}\\*
464e2&120&2&1,1/2&2,1&\\*
\hline\end{tabular}\\
\begin{equation*}
\begin{tikzcd}
464e2 \arrow{r}{2} & \color{red}{464e1}
\end{tikzcd}
\end{equation*}
\end{center}
\begin{center}\begin{tabular}{|c|c|c|c|c|c|}
   \hline&deg&$c_E$&$[L_f:\cL_E]$&$[H_1(E,\ZZ):\pi^*H_1(X_0(N),\ZZ)]$&Matrice\\\hline507a1&312&1&1,1&1,1&\multirow{2}{*}{$\begin{pmatrix}0&7\\0&0\end{pmatrix}$}\\*
507a2&2184&1&7,1&7,1&\\*
\hline\end{tabular}\\
\begin{equation*}
\begin{tikzcd}
\color{red}{507a1} \arrow{r}{7} & 507a2
\end{tikzcd}
\end{equation*}
\end{center}
\begin{center}\begin{tabular}{|c|c|c|c|c|c|}
   \hline&deg&$c_E$&$[L_f:\cL_E]$&$[H_1(E,\ZZ):\pi^*H_1(X_0(N),\ZZ)]$&Matrice\\\hline585f1&192&1&1,1&1,1&\multirow{8}{*}{$\begin{pmatrix}0&2&0&0&0&0&0&0\\0&0&2&0&0&0&0&2\\0&0&0&2&0&0&2&0\\0&0&0&0&2&2&0&0\\0&0&0&0&0&0&0&0\\0&0&0&0&0&0&0&0\\0&0&0&0&0&0&0&0\\0&0&0&0&0&0&0&0\end{pmatrix}$}\\*
585f2&384&1&2,1&2,1&\\*
585f3&768&1&4,1&4,1&\\*
585f5&1536&1&8,1&8,1&\\*
585f7&3072&1&16,1&16,1&\\*
585f8&3072&1&16,1&16,1&\\*
585f6&1536&1&8,1&8,1&\\*
585f4&768&1&4,1&4,1&\\*
\hline\end{tabular}\\
\begin{equation*}
\begin{tikzcd}
\color{red}{585f1} \arrow{r}{2}& 585f2 \arrow{r}{2}\arrow{rd}{2}& 585f3\arrow{rd}{2}\arrow{r}{2}& 585f5\arrow{rd}{2}\arrow{r}{2}& 585f7\\
&& 585f4&585f6&585f8\end{tikzcd}
\end{equation*}
\end{center}
\begin{center}\begin{tabular}{|c|c|c|c|c|c|}
   \hline&deg&$c_E$&$[L_f:\cL_E]$&$[H_1(E,\ZZ):\pi^*H_1(X_0(N),\ZZ)]$&Matrice\\\hline681b1&375&1&1,1&1,1&\multirow{4}{*}{$\begin{pmatrix}0&2&2&0\\0&0&0&0\\0&0&0&0\\2&0&0&0\end{pmatrix}$}\\*
681b3&750&1&2,1&2,1&\\*
681b4&750&1&2,1&2,1&\\*
681b2&750&2&1,1/2&2,1&\\*
\hline\end{tabular}\\
\begin{equation*}
\begin{tikzcd}
681b2\arrow{r}{2} & \color{red}{681b1} \arrow{rd}{2}\arrow{r}{2} & 681b3\\
&&681b4\end{tikzcd}
\end{equation*}
\end{center}
\begin{center}\begin{tabular}{|c|c|c|c|c|c|}
   \hline&deg&$c_E$&$[L_f:\cL_E]$&$[H_1(E,\ZZ):\pi^*H_1(X_0(N),\ZZ)]$&Matrice\\\hline692a1&123&1&1,1&1,1&\multirow{2}{*}{$\begin{pmatrix}0&0\\2&0\end{pmatrix}$}\\*
692a2&246&2&1,1/2&2,1&\\*
\hline\end{tabular}\\
\begin{equation*}
\begin{tikzcd}
692a2 \arrow{r}{2} & \color{red}{692a1}
\end{tikzcd}
\end{equation*}
\end{center}
\begin{center}\begin{tabular}{|c|c|c|c|c|c|}
   \hline&deg&$c_E$&$[L_f:\cL_E]$&$[H_1(E,\ZZ):\pi^*H_1(X_0(N),\ZZ)]$&Matrice\\\hline848d1&84&1&1,1&1,1&\multirow{2}{*}{$\begin{pmatrix}0&0\\2&0\end{pmatrix}$}\\*
848d2&168&2&1,1/2&2,1&\\*
\hline\end{tabular}\\
\begin{equation*}
\begin{tikzcd}
848d2 \arrow{r}{2} & \color{red}{848d1}
\end{tikzcd}
\end{equation*}
\end{center}
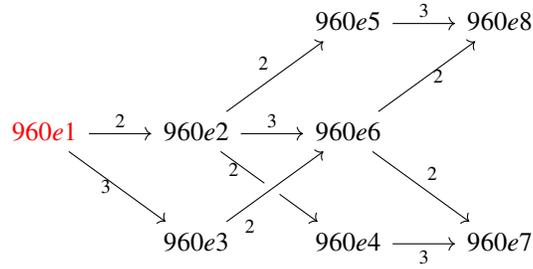
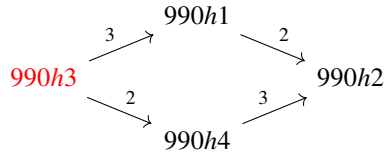
\begin{center}\begin{tabular}{|c|c|c|c|c|c|}
   \hline&deg&$c_E$&$[L_f:\cL_E]$&$[H_1(E,\ZZ):\pi^*H_1(X_0(N),\ZZ)]$&Matrice\\\hline960e1&384&1&1,1&1,1&\multirow{8}{*}{$\begin{pmatrix}0&2&0&0&3&0&0&0\\0&0&2&2&0&3&0&0\\0&0&0&0&0&0&3&0\\0&0&0&0&0&0&0&3\\0&0&0&0&0&2&0&0\\0&0&0&0&0&0&2&2\\0&0&0&0&0&0&0&0\\0&0&0&0&0&0&0&0\end{pmatrix}$}\\*
960e2&768&1&2,1&2,1&\\*
960e5&1536&1&4,1&4,1&\\*
960e4&1536&1&4,1&4,1&\\*
960e3&1152&1&3,1&3,1&\\*
960e6&2304&1&6,1&6,1&\\*
960e8&4608&1&12,1&12,1&\\*
960e7&4608&1&12,1&12,1&\\*
\hline\end{tabular}\\
\begin{equation*}
\begin{tikzcd}[row sep=large]
 && 960e5\arrow{r}{3}&960e8 \\
\color{red}{960e1} \arrow{r}{2}\arrow{rd}[left]{3}&960e2\arrow{rd}[left,near start]{2}\arrow{ru}{2}\arrow{r}{3}&960e6\arrow{rd}{2}\arrow{ru}[right]{2}\\
&960e3\arrow[crossing over]{ru}[below right, very near start]{2}&960e4\arrow{r}[below]{3}&960e7
\end{tikzcd}
\end{equation*}
\end{center}
\begin{center}\begin{tabular}{|c|c|c|c|c|c|}
   \hline&deg&$c_E$&$[L_f:\cL_E]$&$[H_1(E,\ZZ):\pi^*H_1(X_0(N),\ZZ)]$&Matrice\\\hline990h3&1728&1&1,1&1,1&\multirow{4}{*}{$\begin{pmatrix}0&3&2&0\\0&0&0&2\\0&0&0&3\\0&0&0&0\end{pmatrix}$}\\*
990h1&5184&1&3,1&3,1&\\*
990h4&3456&1&2,1&2,1&\\*
990h2&10368&1&6,1&6,1&\\*
\hline\end{tabular}\\
\begin{equation*}
\begin{tikzcd}[row sep=small]
&990h1\arrow{rd}{2} \\
\color{red}{990h3}\arrow{ru}{3}\arrow{rd}{2}&& 990h2 \\
&990h4\arrow{ru}{3} & \end{tikzcd}
\end{equation*}
\end{center}
\begin{center}\begin{tabular}{|c|c|c|c|c|c|}
   \hline&deg&$c_E$&$[L_f:\cL_E]$&$[H_1(E,\ZZ):\pi^*H_1(X_0(N),\ZZ)]$&Matrice\\\hline1089e1&120&1&1,1&1,1&\multirow{2}{*}{$\begin{pmatrix}0&11\\0&0\end{pmatrix}$}\\*
1089e2&1320&1&11,1&11,1&\\*
\hline\end{tabular}\\
\begin{equation*}
\begin{tikzcd}
\color{red}{1089e1} \arrow{r}{11} & 1089e2
\end{tikzcd}
\end{equation*}
\end{center}
\begin{center}\begin{tabular}{|c|c|c|c|c|c|}
   \hline&deg&$c_E$&$[L_f:\cL_E]$&$[H_1(E,\ZZ):\pi^*H_1(X_0(N),\ZZ)]$&Matrice\\\hline1225e1&1680&1&1,1&1,1&\multirow{2}{*}{$\begin{pmatrix}0&37\\0&0\end{pmatrix}$}\\*
1225e2&62160&1&37,1&37,1&\\*
\hline\end{tabular}\\
\begin{equation*}
\begin{tikzcd}
\color{red}{1225e1} \arrow{r}{37} & 1225e2
\end{tikzcd}
\end{equation*}
\end{center}
\begin{center}\begin{tabular}{|c|c|c|c|c|c|}
   \hline&deg&$c_E$&$[L_f:\cL_E]$&$[H_1(E,\ZZ):\pi^*H_1(X_0(N),\ZZ)]$&Matrice\\\hline1849a1&264&1&1,1&1,1&\multirow{2}{*}{$\begin{pmatrix}0&43\\0&0\end{pmatrix}$}\\*
1849a2&11352&1&43,1&43,1&\\*
\hline\end{tabular}\\
\begin{equation*}
\begin{tikzcd}
\color{red}{1849a1} \arrow{r}{43} & 1849a2
\end{tikzcd}
\end{equation*}
\end{center}
\begin{center}\begin{tabular}{|c|c|c|c|c|c|}
   \hline&deg&$c_E$&$[L_f:\cL_E]$&$[H_1(E,\ZZ):\pi^*H_1(X_0(N),\ZZ)]$&Matrice\\\hline1913b1&309&1&1,1&1,1&\multirow{2}{*}{$\begin{pmatrix}0&0\\2&0\end{pmatrix}$}\\*
1913b2&618&2&1,1/2&2,1&\\*
\hline\end{tabular}\\
\begin{equation*}
\begin{tikzcd}
1913b2 \arrow{r}{2} & \color{red}{1913b1}
\end{tikzcd}
\end{equation*}
\end{center}
\begin{center}\begin{tabular}{|c|c|c|c|c|c|}
   \hline&deg&$c_E$&$[L_f:\cL_E]$&$[H_1(E,\ZZ):\pi^*H_1(X_0(N),\ZZ)]$&Matrice\\\hline1936k1&384&1&1,1&1,1&\multirow{2}{*}{$\begin{pmatrix}0&11\\0&0\end{pmatrix}$}\\*
1936k2&4224&1&11,1&11,1&\\*
\hline\end{tabular}\\
\begin{equation*}
\begin{tikzcd}
\color{red}{1936k1} \arrow{r}{11} & 1936k2
\end{tikzcd}
\end{equation*}
\end{center}
\begin{center}\begin{tabular}{|c|c|c|c|c|c|}
   \hline&deg&$c_E$&$[L_f:\cL_E]$&$[H_1(E,\ZZ):\pi^*H_1(X_0(N),\ZZ)]$&Matrice\\\hline2089b1&219&1&1,1&1,1&\multirow{2}{*}{$\begin{pmatrix}0&0\\2&0\end{pmatrix}$}\\*
2089b2&438&2&1,1/2&2,1&\\*
\hline\end{tabular}\\
\begin{equation*}
\begin{tikzcd}
2089b2 \arrow{r}{2} & \color{red}{2089b1}
\end{tikzcd}
\end{equation*}
\end{center}
\begin{center}\begin{tabular}{|c|c|c|c|c|c|}
   \hline&deg&$c_E$&$[L_f:\cL_E]$&$[H_1(E,\ZZ):\pi^*H_1(X_0(N),\ZZ)]$&Matrice\\\hline2145e1&19968&1&1,1&1,1&\multirow{8}{*}{$\begin{pmatrix}0&2&0&0&0&0&0&0\\0&0&2&2&0&0&0&0\\0&0&0&0&0&0&0&0\\0&0&0&0&2&2&0&0\\0&0&0&0&0&0&0&0\\0&0&0&0&0&0&2&2\\0&0&0&0&0&0&0&0\\0&0&0&0&0&0&0&0\end{pmatrix}$}\\*
2145e2&39936&1&2,1&2,1&\\*
2145e4&79872&1&4,1&4,1&\\*
2145e3&79872&1&4,1&4,1&\\*
2145e6&159744&1&8,1&8,1&\\*
2145e5&159744&1&8,1&8,1&\\*
2145e8&319488&1&16,1&16,1&\\*
2145e7&319488&1&16,1&16,1&\\*
\hline\end{tabular}\\
\begin{equation*}
\begin{tikzcd}
\color{red}{2145e1} \arrow{r}{2}& 2145e2 \arrow{r}{2}\arrow{rd}{2}& 2145e3\arrow{rd}{2}\arrow{r}{2}& 2145e5\arrow{rd}{2}\arrow{r}{2}& 2145e8\\
&& 2145e4&2145e6&2145e7\end{tikzcd}
\end{equation*}
\end{center}
\begin{center}\begin{tabular}{|c|c|c|c|c|c|}
   \hline&deg&$c_E$&$[L_f:\cL_E]$&$[H_1(E,\ZZ):\pi^*H_1(X_0(N),\ZZ)]$&Matrice\\\hline2273a1&392&1&1,1&1,1&\multirow{2}{*}{$\begin{pmatrix}0&0\\2&0\end{pmatrix}$}\\*
2273a2&784&2&1,1/2&2,1&\\*
\hline\end{tabular}\\
\begin{equation*}
\begin{tikzcd}
2273a2 \arrow{r}{2} & \color{red}{2273a1}
\end{tikzcd}
\end{equation*}
\end{center}
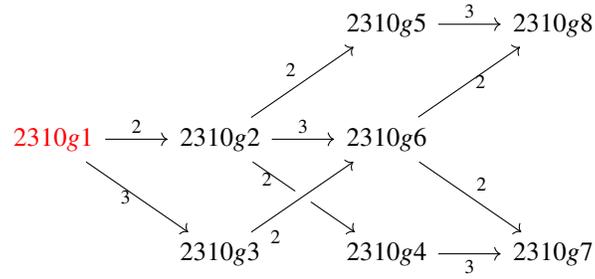
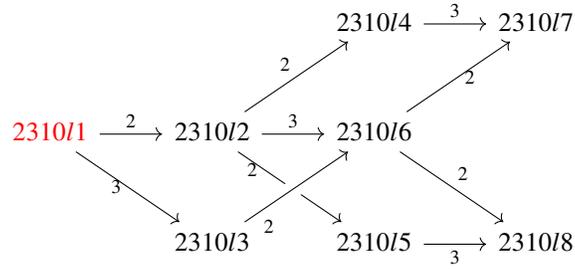
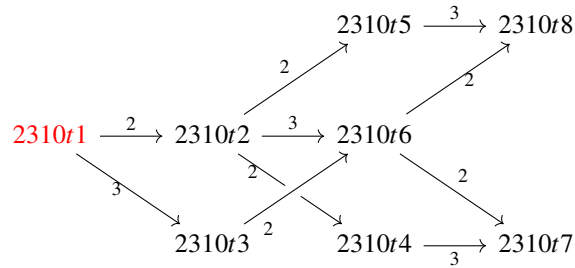
\begin{center}\begin{tabular}{|c|c|c|c|c|c|}
   \hline&deg&$c_E$&$[L_f:\cL_E]$&$[H_1(E,\ZZ):\pi^*H_1(X_0(N),\ZZ)]$&Matrice\\\hline2310g1&13824&1&1,1&1,1&\multirow{8}{*}{$\begin{pmatrix}0&2&0&0&3&0&0&0\\0&0&2&2&0&3&0&0\\0&0&0&0&0&0&3&0\\0&0&0&0&0&0&0&3\\0&0&0&0&0&2&0&0\\0&0&0&0&0&0&2&2\\0&0&0&0&0&0&0&0\\0&0&0&0&0&0&0&0\end{pmatrix}$}\\*
2310g2&27648&1&2,1&2,1&\\*
2310g5&55296&1&4,1&4,1&\\*
2310g4&55296&1&4,1&4,1&\\*
2310g3&41472&1&3,1&3,1&\\*
2310g6&82944&1&6,1&6,1&\\*
2310g8&165888&1&12,1&12,1&\\*
2310g7&165888&1&12,1&12,1&\\*
\hline\end{tabular}\\
\begin{equation*}
\begin{tikzcd}[row sep=large]
 && 2310g5\arrow{r}{3}&2310g8 \\
\color{red}{2310g1} \arrow{r}{2}\arrow{rd}[left]{3}&2310g2\arrow{rd}[left,near start]{2}\arrow{ru}{2}\arrow{r}{3}&2310g6\arrow{rd}{2}\arrow{ru}[right]{2}\\
&2310g3\arrow[crossing over]{ru}[below right, very near start]{2}&2310g4\arrow{r}[below]{3}&2310g7
\end{tikzcd}
\end{equation*}
\end{center}
\begin{center}\begin{tabular}{|c|c|c|c|c|c|}
   \hline&deg&$c_E$&$[L_f:\cL_E]$&$[H_1(E,\ZZ):\pi^*H_1(X_0(N),\ZZ)]$&Matrice\\\hline2310l1&2304&1&1,1&1,1&\multirow{8}{*}{$\begin{pmatrix}0&2&0&0&3&0&0&0\\0&0&2&2&0&3&0&0\\0&0&0&0&0&0&3&0\\0&0&0&0&0&0&0&3\\0&0&0&0&0&2&0&0\\0&0&0&0&0&0&2&2\\0&0&0&0&0&0&0&0\\0&0&0&0&0&0&0&0\end{pmatrix}$}\\*
2310l2&4608&1&2,1&2,1&\\*
2310l4&9216&1&4,1&4,1&\\*
2310l5&9216&1&4,1&4,1&\\*
2310l3&6912&1&3,1&3,1&\\*
2310l6&13824&1&6,1&6,1&\\*
2310l7&27648&1&12,1&12,1&\\*
2310l8&27648&1&12,1&12,1&\\*
\hline\end{tabular}\\
\begin{equation*}
\begin{tikzcd}[row sep=large]
 && 2310l4\arrow{r}{3}&2310l7 \\
\color{red}{2310l1} \arrow{r}{2}\arrow{rd}[left]{3}&2310l2\arrow{rd}[left,near start]{2}\arrow{ru}{2}\arrow{r}{3}&2310l6\arrow{rd}{2}\arrow{ru}[right]{2}\\
&2310l3\arrow[crossing over]{ru}[below right, very near start]{2}&2310l5\arrow{r}[below]{3}&2310l8
\end{tikzcd}
\end{equation*}
\end{center}
\begin{center}\begin{tabular}{|c|c|c|c|c|c|}
   \hline&deg&$c_E$&$[L_f:\cL_E]$&$[H_1(E,\ZZ):\pi^*H_1(X_0(N),\ZZ)]$&Matrice\\\hline2310t1&3456&1&1,1&1,1&\multirow{8}{*}{$\begin{pmatrix}0&2&0&0&3&0&0&0\\0&0&2&2&0&3&0&0\\0&0&0&0&0&0&3&0\\0&0&0&0&0&0&0&3\\0&0&0&0&0&2&0&0\\0&0&0&0&0&0&2&2\\0&0&0&0&0&0&0&0\\0&0&0&0&0&0&0&0\end{pmatrix}$}\\*
2310t2&6912&1&2,1&2,1&\\*
2310t5&13824&1&4,1&4,1&\\*
2310t4&13824&1&4,1&4,1&\\*
2310t3&10368&1&3,1&3,1&\\*
2310t6&20736&1&6,1&6,1&\\*
2310t8&41472&1&12,1&12,1&\\*
2310t7&41472&1&12,1&12,1&\\*
\hline\end{tabular}\\
\begin{equation*}
\begin{tikzcd}[row sep=large]
 && 2310t5\arrow{r}{3}&2310t8 \\
\color{red}{2310t1} \arrow{r}{2}\arrow{rd}[left]{3}&2310t2\arrow{rd}[left,near start]{2}\arrow{ru}{2}\arrow{r}{3}&2310t6\arrow{rd}{2}\arrow{ru}[right]{2}\\
&2310t3\arrow[crossing over]{ru}[below right, very near start]{2}&2310t4\arrow{r}[below]{3}&2310t7
\end{tikzcd}
\end{equation*}
\end{center}
\begin{center}\begin{tabular}{|c|c|c|c|c|c|}
   \hline&deg&$c_E$&$[L_f:\cL_E]$&$[H_1(E,\ZZ):\pi^*H_1(X_0(N),\ZZ)]$&Matrice\\\hline2352j1&1680&1&1,1&1,1&\multirow{2}{*}{$\begin{pmatrix}0&13\\0&0\end{pmatrix}$}\\*
2352j2&21840&1&13,1&13,1&\\*
\hline\end{tabular}\\
\begin{equation*}
\begin{tikzcd}
\color{red}{2352j1} \arrow{r}{13} & 2352j2
\end{tikzcd}
\end{equation*}
\end{center}
\begin{center}\begin{tabular}{|c|c|c|c|c|c|}
   \hline&deg&$c_E$&$[L_f:\cL_E]$&$[H_1(E,\ZZ):\pi^*H_1(X_0(N),\ZZ)]$&Matrice\\\hline3249a1&640&1&1,1&1,1&\multirow{2}{*}{$\begin{pmatrix}0&19\\0&0\end{pmatrix}$}\\*
3249a2&12160&1&19,1&19,1&\\*
\hline\end{tabular}\\
\begin{equation*}
\begin{tikzcd}
\color{red}{3249a1} \arrow{r}{19} & 3249a2
\end{tikzcd}
\end{equation*}
\end{center}
\begin{center}\begin{tabular}{|c|c|c|c|c|c|}
   \hline&deg&$c_E$&$[L_f:\cL_E]$&$[H_1(E,\ZZ):\pi^*H_1(X_0(N),\ZZ)]$&Matrice\\\hline4489a1&1292&1&1,1&1,1&\multirow{2}{*}{$\begin{pmatrix}0&67\\0&0\end{pmatrix}$}\\*
4489a2&86564&1&67,1&67,1&\\*
\hline\end{tabular}\\
\begin{equation*}
\begin{tikzcd}
\color{red}{4489a1} \arrow{r}{67} & 4489a2
\end{tikzcd}
\end{equation*}
\end{center}
\begin{center}\begin{tabular}{|c|c|c|c|c|c|}
   \hline&deg&$c_E$&$[L_f:\cL_E]$&$[H_1(E,\ZZ):\pi^*H_1(X_0(N),\ZZ)]$&Matrice\\\hline14450p1&12240&1&1,1&1,1&\multirow{2}{*}{$\begin{pmatrix}0&17\\0&0\end{pmatrix}$}\\*
14450p2&208080&1&17,1&17,1&\\*
\hline\end{tabular}\\
\begin{equation*}
\begin{tikzcd}
\color{red}{14450p1} \arrow{r}{17} & 14450p2
\end{tikzcd}
\end{equation*}
\end{center}
\begin{center}\begin{tabular}{|c|c|c|c|c|c|}
   \hline&deg&$c_E$&$[L_f:\cL_E]$&$[H_1(E,\ZZ):\pi^*H_1(X_0(N),\ZZ)]$&Matrice\\\hline16641e1&8448&1&1,1&1,1&\multirow{2}{*}{$\begin{pmatrix}0&43\\0&0\end{pmatrix}$}\\*
16641e2&363264&1&43,1&43,1&\\*
\hline\end{tabular}\\
\begin{equation*}
\begin{tikzcd}
\color{red}{16641e1} \arrow{r}{43} & 16641e2
\end{tikzcd}
\end{equation*}
\end{center}
\begin{center}\begin{tabular}{|c|c|c|c|c|c|}
   \hline&deg&$c_E$&$[L_f:\cL_E]$&$[H_1(E,\ZZ):\pi^*H_1(X_0(N),\ZZ)]$&Matrice\\\hline26569a1&59368&1&1,1&1,1&\multirow{2}{*}{$\begin{pmatrix}0&163\\0&0\end{pmatrix}$}\\*
26569a2&9676984&1&163,1&163,1&\\*
\hline\end{tabular}\\
\begin{equation*}
\begin{tikzcd}
\color{red}{26569a1} \arrow{r}{163} & 26569a2
\end{tikzcd}
\end{equation*}
\end{center}
\begin{center}\begin{tabular}{|c|c|c|c|c|c|}
   \hline&deg&$c_E$&$[L_f:\cL_E]$&$[H_1(E,\ZZ):\pi^*H_1(X_0(N),\ZZ)]$&Matrice\\\hline130050em1&1769472&1&1,1&1,1&\multirow{8}{*}{$\begin{pmatrix}0&2&0&0&3&0&0&0\\0&0&2&2&0&3&0&0\\0&0&0&0&0&0&3&0\\0&0&0&0&0&0&0&3\\0&0&0&0&0&2&0&0\\0&0&0&0&0&0&2&2\\0&0&0&0&0&0&0&0\\0&0&0&0&0&0&0&0\end{pmatrix}$}\\*
130050em2&3538944&1&2,1&2,1&\\*
130050em4&7077888&1&4,1&4,1&\\*
130050em5&7077888&1&4,1&4,1&\\*
130050em3&5308416&1&3,1&3,1&\\*
130050em6&10616832&1&6,1&6,1&\\*
130050em7&21233664&1&12,1&12,1&\\*
130050em8&21233664&1&12,1&12,1&\\*
\hline\end{tabular}\\
\begin{equation*}
\begin{tikzcd}[row sep=large]
 && 130050em4\arrow{r}{3}&130050em7 \\
\color{red}{130050em1} \arrow{r}{2}\arrow{rd}[left]{3}&130050em2\arrow{rd}[left,near start]{2}\arrow{ru}{2}\arrow{r}{3}&130050em6\arrow{rd}{2}\arrow{ru}[right]{2}\\
&130050em3\arrow[crossing over]{ru}[below right, very near start]{2}&130050em5\arrow{r}[below]{3}&130050em8
\end{tikzcd}
\end{equation*}
\end{center}

\end{document}